\documentclass[12pt]{amsart}
\usepackage{fullpage,verbatim,amssymb}
\usepackage{hyperref, bbm}
\usepackage{soul}
\usepackage[active]{srcltx}
\usepackage[usenames,dvipsnames]{color}
\hypersetup{colorlinks=true,linkcolor=blue,citecolor=magenta}

\makeatletter
\let\@@pmod\mod
\DeclareRobustCommand{\mod}{\@ifstar\@pmods\@@pmod}
\def\@pmods#1{\mkern4mu({\operator@font mod}\mkern 6mu#1)}
\makeatother

\definecolor{blue}{rgb}{0,0,1}
\definecolor{red}{rgb}{1,0,0}
\definecolor{green}{rgb}{0,.6,.2}
\definecolor{purple}{rgb}{1,0,1}

\long\def\red#1\endred{\textcolor{red}{#1}}
\long\def\blue#1\endblue{\textcolor{blue}{#1}}
\long\def\purple#1\endpurple{\textcolor{purple}{ #1}}
\long\def\green#1\endgreen{\textcolor{green}{#1}}

\newcommand{\ph}{\varphi}
\newcommand{\g}{\gamma}

\newcommand{\scrL}{\mathcal{L}}

\newcommand{\scrF}{\mathcal{F}}

\newcommand{\Z}{\mathbb{Z}}
\newcommand{\N}{\mathbb{N}}

\newcommand{\R}{\mathbb{R}}
\newcommand{\C}{\mathbb{C}}

\newcommand{\HH}{\mathbb{H}}

\def\Im{\operatorname{Im}}
\def\Re{\operatorname{Re}}

\DeclareMathOperator{\SL}{SL}

\newcommand{\sm}{\left(\begin{smallmatrix}}
\newcommand{\esm}{\end{smallmatrix}\right)}
\newcommand{\bpm}{\begin{pmatrix}}
\newcommand{\ebpm}{\end{pmatrix}}

\newtheorem{theorem}{Theorem}
\newtheorem{lemma}[theorem]{Lemma}
\newtheorem{proposition}[theorem]{Proposition}
\newtheorem{corollary}[theorem]{Corollary}
\newtheorem{definition}[theorem]{Definition}

\theoremstyle{remark}

\newtheorem{remark}[theorem]{Remark}

\newtheorem{rmk}{Remark}

\numberwithin{theorem}{section}
\numberwithin{equation}{section}
\numberwithin{rmk}{section}

\title{$L$-values of harmonic Maass forms}
\author{Nikolaos Diamantis}
\address{University of Nottingham}
\email{nikolaos.diamantis@nottingham.ac.uk}
\author{Larry Rolen}
\address{Vanderbilt University}
\email{larry.rolen@vanderbilt.edu}

\begin{document}

\maketitle

\begin{abstract}
    Bruinier, Funke, and Imamoglu have proved a formula for what can philosophically be called the ``central $L$-value'' of the modular $j$-invariant. Previously, this had been heuristically suggested by Zagier. Here, we interpret this ``$L$-value" as the value of an actual $L$-series, and extend it to all integral arguments and to a large class of harmonic Maass forms which includes all weakly holomorphic cusp forms. The context and relation to previously defined $L$-series for weakly holomorphic and harmonic Maass forms are discussed. These formulas suggest possible arithmetic or geometric meaning of $L$-values in these situations. The key ingredient of the proof is to apply a  recent theory of Lee, Raji, and the authors to describe harmonic Maass $L$-functions using test functions.
\end{abstract}

\section{Introduction}\label{intro}
Some of the most important quantities associated to a modular form are its Fourier coefficients and its $L$-values. Their Fourier coefficients give important sequences in number theory, combinatorics, and physics, whereas $L$-values often encode deep geometric and arithmetic information which is the subject of conjectures such as those by Birch--Swinnerton-Dyer and by Beilinson.

Particularly rich insight can be derived when $L$-values appear as Fourier coefficients of other modular forms (e.g. in Shimura correspondence). A situation of similar nature, in the setting of harmonic Maass forms, occurred in \cite{BFI}, where
a harmonic Maass form was constructed based on modular traces and cycle integrals of weakly holomorphic modular forms.
There, some of the Fourier coefficients of the constructed harmonic Maass form were expressed as the ``central value" of an undefined ``$L$-function" of a weakly holomorphic form. That expression allowed for a geometric interpretation of those coefficients.

We recall the expression established in \cite{BFI}, maintaining their notation.  We let
$C_Q$ denote the imaginary axis and let $J(z):=j(z)-744$ be the Hauptmodul for $\operatorname{SL}_2(\mathbb Z)$. By analogy with the Mellin transform expression for the value of a weight $k$ cusp form at $k/2$, the authors of \cite{BFI} define the
``central $L$-value" of $J$ as the regularized integral
\begin{equation}\label{reg}
\int_{C_Q}^{\mathrm{reg}}J(\tau)\frac{d\tau}{\tau}:=
2 \sum_{n \neq 0} a(n) \mathcal {EI}(2 \pi n)
\end{equation}
where $\mathcal {EI}$ is the special function
\begin{equation}\label{EJInt}
\mathcal{EI}(w):=\int_w^{\infty}e^{-t}\frac{dt}{t}=\begin{cases}
E_1(w) \qquad \text{if $w>0$}
\\
-\text{Ei}(-w) \quad \text{if $w<0$}.
\end{cases}
\end{equation}
Here $E_1$ stands for the exponential integral  (see \eqref{E-G} and \eqref{EpAnCont} below for the general definition of $E_s(z)$) and Ei for the ``complementary" exponential integral defined as the Cauchy principal value of the integral (see \S 6.2(i) of \cite{NIST}).
Then, Th. 1.1 of \cite{BFI} implies
\begin{equation}\label{ZagInt}
\int_{C_Q}^{\mathrm{reg}}J(\tau)\frac{d\tau}{\tau}=-\int_i^{i+1}J(\tau)\left
(\psi(\tau)+\psi(1-\tau) \right ) d\tau
=-2\mathrm{Re}\left(\int_i^{i+1}J(\tau)\psi(\tau)d\tau\right)
,
\end{equation}
where $\psi(z)$ is the Euler Digamma function $\psi(z):=\Gamma'(z)/\Gamma(z).$
That expression generalized a formula suggested to the authors of \cite{BFI} by D. Zagier.
More recently, a similar formula was proved for general harmonic Maass forms in \cite{AS} in the context of their work
on a Shintani correspondence of harmonic Maass forms. In that work too, the crucial element was an explicit characterization of the ``central value" of a hypothetical ``$L$-function" attached to a harmonic Maass form.

As pointed out in \cite{BFI} and \cite{AS}, the reason that the above explicit formulas could be considered as $L$-values only by analogy is that, at the time, there was no systematic construction of $L$-series for harmonic Maass forms. Recently, however, the authors jointly with M. Lee and W. Raji have defined and studied actual $L$-series for general harmonic Maass forms (\cite{DLRR}). In this paper, we will use the theory of \cite{DLRR} to interpret the expressions established in \cite{BFI} and \cite{AS} in the framework of a properly defined $L$-series. We will further extend those explicit expressions beyond the
``central value" to include, in particular, all integer points.

The main results in their full generality will be given in Sections~\ref{mainS} and \ref{appl}. Here, we offer a special case for weakly holomorphic modular forms. To this end, let
\begin{equation}\label{FourInt}
f(z) = \sum_{\substack{n=-n_0 \\  n \ne 0}}^\infty a_f(n) e^{2\pi inz}
\end{equation} be an element of $S_k^!(N)$, the space of weight $k$ weakly holomorphic cusp forms, i.e.,  the space of weakly holomorphic modular forms with vanishing constant terms at the cusps. For each $s, w \in \C$, let
$$\varphi_s^w(t):=\mathbf 1_{[1, \infty)}(t) e^{-wt}t^{s-1}, \qquad \text{for $t>0$.}$$
where $\mathbf 1_S$ denotes the characteristic function of $S \subset \mathbb R$.
Then we set
\begin{equation} \label{LfunInt}
L_f(\varphi_s^w):=\sum_{\substack{n=-n_0 \\  n \ne 0}}^\infty a_f(n) (\mathcal L \varphi_s^w)(2 \pi n)
\end{equation}
where $\mathcal L(g)$ denotes the standard Laplace transform of $g$ (see \eqref{e:Laplace_trans} below). With this definition,
we have the following result, where $$\zeta(s, a, z):=\sum_{m=0}^{\infty} e^{2 \pi i m a}(z+m)^{-s}$$
is the Lerch zeta function.
\begin{theorem}\label{maincor} Let $f \in S^!_k(N)$. Then, for each $s \in \mathbb R$ and each $w$ with $\Im(w)>0$ 
%and $\Re(w) \ge 0$
we have
\begin{equation*}
L_f(\varphi_s^w)=
i^{-s}\int_{i}^{i+1} f(z)e^{i w z} \zeta \left (1-s, \frac{w}{2\pi}, z \right ) dz.
\end{equation*}
\end{theorem}
As applications, we show Zagier-like formulas for forms of all weights $k$ and integer $L$-values $s$ which generalize \eqref{ZagInt} of \cite{BFI}. To describe this, for $\varphi_s^w$ as above, set
\begin{equation}\label{L(f}
L^*(f, s):=L_f(\varphi_s^0)= \sum_{\substack{n=-n_0 \\  n \ne 0}}^\infty a_f(n) E_{1-s}(2 \pi n),
\end{equation}
(the star superscript added to indicate the analogy with the ``completed" version of the classical $L$-series, rather than with the $L$-series itself).
Then, in Section \ref{appl} we will prove the following formula, where $\zeta^*(a, z)$ denotes the constant term in the Laurent expansion at $s=a$ of the Hurwitz zeta function
$ \zeta(s, z)=\zeta(s, 0, z)$, i.e. $\zeta^*(a, z)$ equals $\zeta(a, z)$, if $a \neq 1$, and
$-\psi(z)$, when $a=1$.
\begin{theorem}
\label{mainInt*} Let $f\in S_k^!(N)$. Then for each $m \in \mathbb Z$, we have
$$L^*(f, m)=i^{-m}\int_{i}^{i+1} f(z) \zeta^* \left (1-m, z\right )dz.$$
\end{theorem}
\begin{rmk} As explained in Remark \ref{cconj} below, since $\zeta^*(1, a)=-\psi(a)$,
Theorem~\ref{mainInt*} directly implies \eqref{ZagInt}.
\end{rmk}
The geometric interpretation of the expression \eqref{ZagInt} established in \cite{BFI}, combined with the systematic approach to $L$-series of harmonic Maass forms of \cite{DLRR} that we apply here, suggests a deeper arithmetic meaning of the $L$-values considered which is worth studying further.

\section{Context and motivation}
As mentioned above, the Fourier coefficients and $L$-values of classical holomorphic modular forms are fundamental invariants.
For concreteness, let $f(z)$ be a cusp form of weight $k$, having Fourier expansion $f(z)=\sum_{n\geq1}a_f(n)q^n$ with $q:=e^{2\pi i z}$. Then the $L$-series for $f$ is defined for $\Re (s)>k+1/2$ by
\[
L(f,s):=\sum_{n\geq1}\frac{a_f(n)}{n^s}.
\]
This has a functional equation relating the values at $s$ and $k-s$, and an analytic continuation to $\mathbb C$. The critical $L$-values may also be given by the following integral expressions:
\begin{equation}\label{Lintegral}
L(f,j+1)=\frac{(2\pi)^{j+1}}{j!}\int_0^{\infty}f(it)t^jdt,\quad\quad (1\leq j\leq k-1).
\end{equation}
These integrals are often called the {\it periods} of $f$. While there is a vast literature on the periods of holomorphic modular forms, little is known about non-holomorphic extensions. However, there is growing evidence that $L$-series in such cases arise in important connections with physics, and that an arithmetic theory may be lurking. In particular, modular forms which have exponential growth near the cusps play an increasingly important role in applications. Often, physical theories require one to allow such singularities. In this situation, integrals such as the right hand side of \eqref{Lintegral} no longer make sense. The concept of  {\it renormalization} in physics requires mollifying divergent integrals
(\cite{HarveyMoore}) and it was a key ingredient in  Borcherds' famous paper on automorphic forms with singularities on Grassmannians \cite{Borcherds}. Borcherds' work
was pivotal for the field of {\it harmonic Maass forms}, including Bruinier and Ono's construction of generalized Borcherds products \cite{BruinierOno}. This has fundamental applications, for example 
the construction of harmonic Maass forms which interpolate central critical $L$-values and $L$-derivatives of elliptic curves.
%, that is, assuming BSD, their ranks up to rank $2$.

\iffalse
As one more example, in an important series of papers \cite{BrownI,BrownII,BrownIII}, Brown studied functions arising from the study of amplitudes in string theory. Inspired by specific examples, Brown defined the space of
%he called the
{\it modular iterated integrals}. In the first two papers, he dealt with functions of moderate growth, but this led to theoretical limitations from the physics perspective. In his third paper, he  showed that the natural remedy is to allow exponential growth.
%Intuitively,  Brown's space is a ``hybrid'' one with the growth properties of harmonic Maass forms, but with more complicated relationships under the Laplace operators.
Brown also gave hints of {\it motivic} structure of his spaces.
Recently, Drewitt and the first author \cite{DD} developed a theory of period polynomials for modular iterated integrals and used it to prove a {\it Manin period theorem}-style result on algebraicity.  
\fi

Our work aims to begin a more systematic study of modular $L$-values when there is exponential growth at the cusps.
 Recently,  Lee, Raji, and the authors developed \cite{DLRR} a new theory of  $L$-series in such cases.
 %, in particular, for weakly holomorphic modular forms and harmonic Maass forms.
 The key idea was to use a distributional definition, rather than think of $L$-series as functions on $\C$.
% ; this is analogous to important recent work of Miller and Schmid on automorphic distributions . This approach of \cite{DLRR} circumvents many issues of divergence, and allows one to recover $L$-functions previously studied by Bringmann, Fricke, and Kent~\cite{BFK}. In particular, the previous attempts at defining $L$-functions in the setting of weakly holomorphic modular forms lacked the same
This distributional definition allowed the development of structures that are critically important for holomorphic modular forms, e.g. functional equations, converse theorems, Voronoi-type summation formulas. Here, we apply the approach of \cite{DLRR} to study periods.
 
 Several early hints on the structure of periods of weakly holomorphic modular forms can be found in the literature.
 Much has been written about the traces of singular moduli for weakly holomorphic forms, initiated by Zagier \cite{ZagierTraces} and related to the work of Borcherds \cite{BorcherdsInfProd}. For example, Zagier showed that the generating function for the traces $tr_d(J)$ of $J(z)$ at discriminant $d<0$ CM points is the weakly holomorphic modular form
 \[
 -q^{-1}+2-248q^3+492q^4+\ldots\in M^!_{\frac32}(\Gamma_0(4)).
 \]
 The analogous situation of {\it positive} discriminants gives rise to the theory of {\it cycle integrals}, which Kohnen \cite{Kohnen} and Kohnen-Zagier \cite{KohnenZagier} showed to be closely tied to
 certain period integrals of holomorphic modular forms.
 An important work of Duke, Imamoglu, and T\'oth \cite{DIT} studied the cycle integrals of $J(z)$. More recently, the ``square discriminant'' cases were studied by Bruinier, Funke, and Imamoglu \cite{BFI}. In the case of $J(z)$, this allowed them to prove a formula, discovered heuristically by Zagier. This gave a cycle integral of $J(z)$ in discriminant $1$, which philosophically can be thought of as a critical $L$-value of $J(z)$:
\[
``L(J,0)"\ =\ 2\int_1^{\infty}J(z)\frac{dz}{z}=2\sum_{n\neq0}a_J(n)\int_{2\pi n}^{\infty}e^{-t}\frac{dt}{t}=-2\Re\left(\int_i^{i+1}J(z)\psi(z)dz\right).
\]

This work recasts the above formula in terms of genuine $L$-series associated with weakly holomorphic forms according to the theory developed in \cite{DLRR}. As outlined in Sect. \ref{intro}, this leads to an extension to all integer (and, in some cases, non-integer) values, and, further, to general harmonic Maass forms, thereby re-interpreting an analogous formula of \cite{AS}.

Our approach also sheds light to an earlier version of an $L$-series of weakly holomorphic modular forms introduced in \cite{BFK} and applied to various settings (e.g. \cite{BDE, DD}). That version, for a weight $k$ holomorphic cusp form $f$ with Fourier expansion \eqref{FourInt},
was given by
\begin{equation}\label{LfInt}
\widetilde L_f(s):= \sum_{\substack{n=-n_0 \\  n \ne 0}}^\infty a_f(n) E_{1-s}(2 \pi n)+
i^k \sum_{\substack{n=-n_0 \\  n \ne 0}}^\infty a_f(n) E_{1-k+s}(2 \pi n).
\end{equation}
 We observe that $\widetilde L_f(s)$, when $k=s=0,$ is similar to \eqref{reg} except that the former has been symmetrized to ensure it satisfies a functional equation $s \to k-s$. Indeed, the first series in the RHS of \eqref{LfInt}, i.e., the function $L^*(f, s)$ of Th. \ref{mainInt*}, does not satisfy a functional equation. Instead, the behavior of this series (and, thus, of \eqref{reg}) under the transformation $s \to k-s$ is fully characterized once it is incorporated into the framework of the $L$-series defined by \eqref{LfunInt}. Interpreted in this way, it does satisfy a functional equation which is stated and proved in Prop. \ref{FE} below. For this reason, we consider as our principal object the $L$-series $L_f(\varphi_s^w)$ and not $L^*(f, s)$. The latter is mainly introduced to interpret our results in the setting of \cite{BFI} and \cite{AS} which were the inspiration of our work.

Another advantage of our method as applied here is that it formalizes the regularization processes employed in both \cite{BFI} and \cite{BFK} (and \cite{BDE}) thus benefiting from the insight gained by their use. From this standpoint, regularization amounts to evaluating $L$-series at \emph{families of mappings} instead of specific values, which, as seen in the present note, allows for greater flexibility and for understanding the behavior under the transformation $s \to k-s.$

The structure of the note is as follows. In the next section, we review the theory of $L$-series associated with harmonic Maass cusp forms as introduced in \cite{DLRR}. The main theorem (Th. \ref{main}) is formulated and proved in Section \ref{mainS}. In Section \ref{appl}, we apply Th. \ref{main} to retrieve \eqref{ZagInt} of \cite{BFI} and extend it to other integer values. However, from our perspective, the ``correct" viewpoint is to consider Th. \ref{main} as the actual extension of \eqref{ZagInt} of \cite{BFI}. The results of Sect. \ref{appl} may appear more similar in form to \eqref{ZagInt}, but they provide only part of the information contained in Th. \ref{main}. Finally, in Section \ref{comp}, we extend \eqref{ZagInt} of \cite{BFI} in a different direction, by evaluating our $L$-series at compactly supported test functions and deriving a formula analogous to \eqref{ZagInt}.

\section*{Acknowledgments} We thank Kathrin Bringmann and Jorma Louko for very useful comments and feedback on the manuscript.
The first author is partially supported by EPSRC grant EP/S032460/1. The second author is grateful for support from a grant from the Simons Foundation (853830, LR), support from a Dean’s Faculty Fellowship from Vanderbilt University, and to the Max Planck Institute for Mathematics in Bonn for its hospitality and financial support.

\section{Harmonic Maass cusp forms and their $L$-series}\label{Prelim}
We recall the basic definitions and properties of harmonic Maass modular forms. For simplicity, we will focus only on cuspidal forms in this paper. These basic facts can all be found in Bruinier-Funke's work \cite{BF}; see also \cite{HMFBook} for a general reference on harmonic Maass forms.

Let $\HH$ denote the complex upper half-plane and set
$$\mathbb{H}^+:=\{w \in \mathbb H; \Re(w) > 0\} \qquad \text{and} \, \, \,  \mathbb{H}^+_0:=\{w \in \mathbb H; \Re(w) \ge 0\}.$$
For $k \in \Z$, we consider the action $|_k$ of $\SL_2(\R)$ on smooth functions $f\colon \HH \to \C$, given by
\begin{equation}
(f|_k\gamma)(z):= j(\gamma, z)^{-k} f(\gamma z), \qquad \text{for $\gamma=\bpm a & b \\  c & d \ebpm \in$ SL$_2(\R)$}.
\end{equation}
Here $\gamma z = \frac{az+b}{cz+d}$ is the M\"obius transformation
and $j(\gamma, z):=cz+d.$
We also let $\Delta_k$ denote the weight $k$ hyperbolic Laplacian on $\HH$ given by
\begin{equation}
\Delta_k:=-4y^2 \frac{\partial}{\partial z} \frac{\partial}{\partial \bar z}+2iky\frac{\partial}{\partial \bar z},
\end{equation}
where $z=x+iy$ with $x,y\in\R$.

With this notation we state the following.
\begin{definition}\label{hmf} Let $N \in \N$. A \emph{harmonic Maass form of weight $k$ for $\Gamma_0(N)$} is a smooth function $f\colon \HH \to \C$
such that:
\begin{enumerate}
\item[i).] For all $\gamma \in \Gamma_0(N),$ we have
$
f|_k \gamma=f.
$
\item[ii).] We have
$
\Delta_k(f)=0.
$
\item[iii).] For each $\gamma=\sm * & * \\ c & d \esm \in \SL_2(\Z)$, there is a polynomial $P(z) \in \C[q^{-1}],$ such that
\begin{equation*}
f(\g z)(cz+d)^{-k}-P(z)=O(e^{-\varepsilon y}), \qquad \text{as $y \to \infty$, for some $\varepsilon>0.$}
\end{equation*}
\end{enumerate}
We let $H_k(N)$ denote the space of harmonic Maass forms with weight $k$ for $\Gamma_0(N)$.
\end{definition}
To describe the Fourier expansions of the elements of $H_k(N)$, we recall the definition and the asymptotic behavior of the incomplete Gamma function.
For $r, z\in \C$ with $\Re(r)>0$, we define the incomplete Gamma function as
\begin{equation}
\Gamma(r,z) := \int_{z}^\infty e^{-t}t^{r}\, \frac{dt}{t}.
\end{equation}
When $z\neq 0$, $\Gamma(r, z)$ is an entire function of $r$ (see \cite{NIST} \S 8.2(ii)).
We note the asymptotic relation for $x \in \R$ (see (8.11.2) of \cite{NIST}):
\begin{equation}\label{asym}
\Gamma(s, x) \sim x^{s-1}e^{-x} \qquad \text{as $|x| \to \infty.$ }
\end{equation}

With this notation we can state the following theorem.
 \begin{theorem}[\cite{BF}] Let $k \in \Z.$ For each $f \in H_k(N)$, we have
 \begin{equation}\label{FourEx0}
f(z) = \sum_{\substack{n \ge -n_0}} a_f(n) e^{2\pi inz}+ \sum_{\substack{n < 0}} b_f(n) \Gamma(1-k, -4 \pi n y)e^{2 \pi i nz}
\end{equation}
for some $a_f(n), b_f(n) \in \C$ and $n_0 \in \mathbb N.$
Analogous expansions hold at the other cusps.
 \end{theorem}
 \noindent
 The first sum is sometimes called the ``holomorphic part" of $f$ and the second, the ``non-holomorphic part."
 
 The subspace we will be dealing with, for simplicity, in this paper is the space
 $HC_k(N)$ of \emph{harmonic Maass cusp forms of weight $k$ and level $N$}. It consists of $f \in H_k(N)$ which have vanishing constant terms at all cusps.
 Another subspace of particular importance is the space $S_k^!(N)$ of \emph{weakly holomorphic cusp forms with weight $k\in \Z$ and level $N$}. It consists of $f \in HC_k(N)$ that are holomorphic on $\HH.$

The growth of the coefficients $a_f(n), b_f(n)$ of $f$ in \eqref{FourEx0} is given by
\begin{equation} \label{coeffbound}
a_f(n), b_f(-n)=O \left (e^{C_f \sqrt{n}}\right ),
\qquad \text{as $n \to \infty$ for some $C_f>0$}.
\end{equation}

We next recall an important mapping sending forms of weight $k$ to forms of weight $2-k$: Let $f$ be an element of $HC_k$ with Fourier expansion \eqref{FourEx0} and set
$$(\xi_kf)(z):=2iy^k\overline{\frac{\partial f}{\partial \bar z}}.$$ Then, this induces a surjective map
$$\xi_k\colon HC_k(N) \twoheadrightarrow S_{2-k}(N),$$
and the output cusp form has Fourier expansion given explicitly by
\begin{equation}\label{xi}
(\xi_kf)(z)=-\sum_{n<0} (-4 \pi n)^{1-k} \overline{b_f(n)} e^{-2 \pi i n z}.
\end{equation}
An implication of this fact is that, when $k \ge 2$, $\xi_kf$ vanishes (since then $S_{2-k}(N)=\{0\}$) and thus, $f$ is a weakly holomorphic cusp form.

Finally, $f^c$ will denote the harmonic Maass form obtained by taking the complex conjugate of the
coefficients of $f$, i.e.,
$$f^c(z):=\overline{f(-\bar z)}.$$

We will now  recall the $L$-series attached to harmonic Maass cusp forms as defined in \cite{DLRR}.  We require some additional definitions to describe the set-up.
Let $C(\R, \C)$ be the space of piece-wise smooth complex-valued functions on $\R$. Let $\scrL$ be the Laplace transform mapping each $\varphi \in C(\R, \C)$ to
\begin{equation}\label{e:Laplace_trans}
(\scrL \varphi)(s):=\int_0^{\infty} e^{-s t} \varphi(t)dt    
\end{equation}
for each $s \in \C$ for which the integral converges absolutely.

For each function $f: \HH \to \C$ given by an absolutely convergent series of the shape
 \begin{equation}\label{FourEx}
f(z) = \sum_{\substack{n \ge -n_0 \\ n \ne 0}} a_f(n) e^{2\pi inz}+ \sum_{\substack{n < 0}} b_f(n) \Gamma(1-k, -4 \pi n y)e^{2 \pi i nz}
\end{equation}
for some $k \in \Z$ and a $n_0 \in \N$, let $\scrF_f$ be the space of functions $\varphi\in C(\R, \C)$ such that the following series converges:  
\begin{equation}\label{Ff}
\sum_{\substack{n \ge -n_0 \\ n \neq 0}} |a_f(n)| (\scrL|\varphi|)\left (2 \pi n\right )
+ \sum_{n<0} |b_f(n)|\left (4\pi |n|\right )^{1-k}\int_0^{\infty}\frac{(\scrL |\varphi_{2-k}|)\left (-2\pi n(2t+1)\right )}{(1+t)^{k}}dt.
\end{equation}

\begin{rmk}
For the functions $f$ we will be considering, the space $\scrF_f$ contains the compactly supported functions.
\end{rmk}

A useful expression for the ``non-holomorphic'' part of the series is (cf. (4.14) of \cite{DLRR}):
\begin{equation}\label{nonhol}
\left (-4\pi n \right )^{1-k}\int_0^{\infty}\frac{\scrL\varphi_{2-k}
 \left (-2\pi n(2t+1)\right )}{(1+t)^k} dt=
 \int_0^{\infty}\Gamma\left(1-k, -4 \pi n y\right)e^{-2 \pi n y}\ph(y)dy.
\end{equation}
With this notation, we are now able to define our $L$-series and recall some basic facts on them, which were given in \cite{DLRR}.
\begin{definition}\label{def:Lf}
Let $f$ be a function on $\HH$ given by
the Fourier expansion \eqref{FourEx}.
The $L$-series of $f$ is defined to be the map
$L_f\colon \mathcal F_f \to \C$ such that, for $\varphi\in \scrF_{f}$,
\begin{equation*}
\begin{aligned}
L_f(\ph)=&\sum_{n \ge -n_0} a_f(n) (\scrL \ph)(2 \pi n)
\\ &+ \sum_{n<0} b_f(n)(-4\pi n)^{1-k}\int_0^{\infty}\frac{(\scrL \varphi_{2-k})(-2\pi n(2t+1))}{(1+t)^k} dt.
\end{aligned}
\end{equation*}
\end{definition}

The analogue of Mellin transform expression for this $L$-series is the following.
\begin{lemma}\label{lem:LfLf'_int}
Let $f$ be a function on $\HH$ as a series in \eqref{FourEx}.
For $\varphi\in \scrF_f$, the $L$-series $L_f(\varphi)$ can be written as
\begin{equation*}
L_f(\varphi)=\int_0^\infty f(iy) \varphi(y) dy.
\end{equation*}
\end{lemma}

Before stating the functional equation of  $L_f$ for $f \in HC_k(N)$, we introduce the action of a Fricke involution in this setting.
For each $a \in \Z, $ $M \in \mathbb N$ and $\ph\colon \R_{>0} \to \C$, we define
\begin{equation}\label{actionR}
(\ph|_{a}W_M)(x):=(Mx)^{-a} \ph\left(\frac{1}{Mx}\right) \qquad \text{for all $x>0$}.
\end{equation}
\begin{rmk}
The reader should note that there is a change in sign convention from earlier in this paper, when  $W_M$ was considered to act on functions on $\mathbb H$.
\end{rmk}
With this notation in mind, the final result about our $L$-series that we require is the following functional equation, proved, in more generality, in \cite{DLRR} (Th. 4.3).
\begin{theorem}\label{DThalf}
Fix $k\in \Z$ and $N\in \N$.
Suppose that $f$ is an element of $HC_k(N)$ with expansion \eqref{FourEx}.
Consider the map $L_{f}\colon  \scrF_{f}\to \C$ given by Definition~\ref{def:Lf}.
Set
\begin{equation*}
g:=f|_kW_{N}    
\end{equation*}
and
\begin{equation*}
\scrF_{f, g} := \left\{\varphi\in \scrF_{f}
\;:\; \varphi|_{2-k} W_N \in \scrF_{g} \right\}.  
\end{equation*}
Then $\scrF_{f, g}\neq \{0\}$ and, for each $\ph \in \scrF_{f, g}$ we have
\begin{equation*}
L_{f} (\ph) =i^k N^{1-k/2} L_{g}(\ph|_{2-k}W_N),  
\end{equation*}
\end{theorem}

\section{The main formula}\label{mainS}
The basic definitions and structure of harmonic Maass cusp forms and their $L$-series having been set up, we are nearly in position to begin deriving our main result, Theorem~\ref{main}. Firstly, however, we require  some notation and basic facts about the generalized exponential integral.

For $\Re(z)>0$, we define the generalized exponential integral by
\begin{equation}\label{E-G}
E_s(z):=z^{s-1}\Gamma(1-s,z)=\int_{1}^{\infty}\frac{e^{-zt}}{t^{s}} dt
\end{equation} The function $E_s(z)$ has an analytic continuation to $\mathbb C - (-\infty, 0]$ as a function of $z$ to give the \emph{principal branch} of $E_s(z)$. With (8.19.8) and (8.19.10) of \cite{NIST}, this analytic continuation can be given by the formula:
\begin{equation}\label{EpAnCont}
E_s(z)= \begin{cases} z^{s-1}\Gamma(1-s)-\sum\limits_{0\leq k}\frac{(-z)^k}{k!(1-s+k)} \qquad &\text{for $s \in \mathbb C -\mathbb N$},\\
\frac{(-z)^{s-1}}{(s-1)!}(\psi(s)-\text{Log}(z))-\sum\limits_{0\leq k \neq s-1}\frac{(-z)^k}{k!(1-s+k)} \qquad &\text{for $s \in \mathbb N$}.
\end{cases}
\end{equation} Since the two series on the right hand side give entire functions, \eqref{EpAnCont} allows us to continuously extend each function $E_s(z)$ to $\mathbb R_{<0}$, once we select a branch for the (implied) logarithm. We will fix this branch to be the principal one.  

The function $E_s(z)$ differs from $\mathcal {EI}(z)$ only in the values in $\mathbb R_{<0}.$ Specifically, by (6.2.7) \cite{NIST},
for $n<0$, we have
\begin{equation}\label{EJ-E_1}
\mathcal{EI}(2 \pi n )=-\text{Ei}(-2 \pi n)=
\text{Ein}(2 \pi n)-\text{Log}(-2 \pi n)-\gamma=E_1(2 \pi n)+\pi i,
\end{equation}
where $\mathrm{Ein}(z):=\int_0^z\frac{1-e^{-t}}{t}dt$ is the complementary exponential integral.
 By (8.11.2) of \cite{NIST}, we also have the bound \begin{equation}\label{boundE}E_s(z)=O(e^{-z}), \qquad \text{
as $z \to \infty$  in the wedge $\arg(z)<3 \pi/2.$}
\end{equation}

Now for $k \in \mathbb Z$  let $f\in HC_k(N)$ be a  harmonic Maass cusp forms of weight $k$ for $\Gamma_0(N)$. Set
$g:=f|_kW_N.$
Exactly as in the classical case, we can show that $g \in HC_k(N)$
and therefore, $g$ will have a Fourier expansion of the form \eqref{FourEx}.
We then define
$$\varphi_s^w(t):=\mathbf 1_{[1, \infty)}(t) e^{-wt}t^{s-1}, \qquad \text{for $t>0$}.$$
Using this, we show the following basic lemma which shows the suitability of our test function for convergence.
\begin{lemma}\label{Ffg} With the above notation, for each $w$ with $x:=\Re(w)>\max (2 \pi n_0, C_g^2N/(2\pi))$ and $s \in \mathbb C$, the map $\varphi_s^w\colon\mathbb R_{>0} \to \mathbb C$ belongs to $\mathcal F_{f, g}.$
\end{lemma}
\begin{proof} For each $n \ge -n_0$ we have
$$\mathcal L(|\varphi_s^w|)(2\pi n)=\int_1^{\infty} e^{-2\pi nt-xt}t^{\sigma-1}dt=E_{1-\sigma}(2 \pi n+x),$$
where $\sigma:=\Re (s).$
Thus, with \eqref{boundE} and \eqref{coeffbound}, the series
\begin{equation*}
\sum_{\substack{n \ge -n_0}} |a_f(n)| \scrL(|\varphi_s^w|)\left (2 \pi n\right )
\end{equation*}
converges. Furthermore, we have $|(\varphi_s^w)_{2-k}|=\varphi_{\sigma-k+1}^x=(\varphi_{\sigma}^x)_{2-k}$ and hence, by \eqref{asym}, we have
\begin{equation}
\begin{aligned}
%\left (4\pi |n|\right )^{1-k}\int_0^{\infty}\frac{(\scrL |(\varphi_s^w)_{2-k}|)\left (-2\pi n(2t+1)\right )}{(1+t)^{k}}dt=
%\left (4\pi |n|\right )^{1-k}\int_0^{\infty}\frac{(\scrL (\varphi_{\sigma+1-k}^x)_{2-k})\left (-2\pi n(2t+1)\right )}{(1+t)^{k}}dt
\int_1^{\infty}\Gamma\left(1-k, -4 \pi n y\right)e^{-(2 \pi n y+x)y}y^{\sigma+1}dy
%\\
&\ll \int_1^{\infty}e^{-(2 \pi |n|+x)y}y^{\sigma-1-k}dy
\\&=O\left(e^{-C_1 |n|}\right), \quad(\text{as } n \to -\infty).
%= \left (\frac{-4\pi n}{M} \right )^{1-k}\int_0^{\infty}\frac{\scrL\varphi_{2-k}
 %\left (\frac{-2\pi n(2t+1)}{M}\right )}{(1+t)^k} dt.
 \end{aligned}
 \end{equation}
Using \eqref{coeffbound} and \eqref{nonhol}, we deduce that the second series of \eqref{Ff} converges. Therefore, $\varphi_s^w \in \mathcal F_f$.
 
 On the other hand,
\begin{equation}\label{inv}(\varphi_s^w|_{2-k}W_N)(t)=\varphi_s^w(1/(Nt))(Nt)^{k-2}=\mathbf 1_{(0, 1/N]}(t) e^{-w/(Nt)}(Nt)^{k-s-1}, \qquad \text{for $t>0$}.
\end{equation} Hence,
$|\varphi_s^w|_{2-k}W_N|=|\varphi_{\sigma}^x| \, |_{2-k}W_N$. Therefore, we have
$$\mathcal L \left (\Big |\varphi_s^w|_{2-k}W_N \Big | \right )(2\pi n)=\int_0^{1/N} e^{-2\pi nt-\frac{x}{Nt}}t^{k-1-\sigma}dt=(x/N)^{k-\sigma}\int_x^{\infty} e^{-t-\frac{2\pi nx}{Nt}}t^{\sigma-k-1}dt.$$
It is clear that the final integral converges for each $n \in \mathbb Z$ and that, if $n>0$, it is bounded from above by
$$(x/N)^{k-\sigma}\int_0^{\infty} e^{-t-\frac{2\pi nx}{Nt}}t^{\sigma-k-1}dt=2 \left (\frac{2 \pi nN}{x}\right )^{\frac{\sigma-k}{2}}K_{k-\sigma}\left(2 \sqrt{\frac{2 \pi n x}{N}}\right)$$
where $K_s(z)$ is the standard K-Bessel function. Then, using (10.40.2) of \cite{NIST}, we deduce that, for $n>0$,
$$\mathcal L\left (\Big |\varphi_s^w|_{2-k}W_N \Big | \right )(2\pi n) \ll e^{-2\sqrt{\frac{2 \pi n x}{N}}}n^{\frac{\sigma-k-1}{2}}.$$
Using the bound \eqref{coeffbound} and that $x>C_g^2N/(2\pi)$, we find that
the series
\begin{equation*}\label{Fg}
\sum_{\substack{n \ge -n_0}} |a_g(n)| \scrL \left (\Big | \varphi_s^w|_{2-k}W_N \Big | \right )\left (2 \pi n\right )
\end{equation*}
converges. Further, by \eqref{inv} we
have
$$\left | \left( \varphi_s^w|_{2-k}W_N\right )_{2-k}(t)\right |=
\mathbf 1_{[0, 1/N)}(t)e^{-x/(Nt)}N^{k-\sigma-1}t^{-\sigma}$$ and
thus, for $n<0,$ we compute
\begin{multline}
%\begin{aligned}
\int_0^{\infty}\frac{(\scrL |(\varphi_s^w|_{2-k}W_N)_{2-k}|)\left (-2\pi n(2t+1)\right )}{(1+t)^{k}}dt\\
=\int_0^{\infty}\frac{1}{(1+t)^{k}} \int_0^{\frac{1}{N}} N^{k-\sigma-1}y^{-\sigma}e^{-\frac{x}{Ny}-2 \pi |n| (2t+1)y}dydt
\\=\int_0^{\infty}\frac{N^{k-2}x^{1-\sigma}}{(1+t)^{k}} \int_x^{\infty} y^{\sigma-2}e^{-y-\frac{2 \pi |n| (2t+1)x}{Ny}}dydt
%\end{aligned}
 \end{multline}
 where, for the inner integral we applied the change of variables
 $y \to x/(Ny).$ The inner integral is bounded from above by
 $$2 \left (\sqrt{\frac{2 \pi |n| (2t+1)x}{N}} \right )^{\sigma-1} K_{1-\sigma}\left (\sqrt{\frac{8 \pi |n| (2t+1)x}{N}} \right ) \ll  (|n|(2t+1))^{\frac{2 \sigma-3}{4}}e^{-\sqrt{\frac{8 \pi x}{N}} \sqrt{|n|(2t+1)}}$$
 as $n \to -\infty$, where we also made use of (10.40.2) of \cite{NIST}. The basic inequality
 $$-\sqrt{|n|(2t+1)} \le \frac{-1}{2}\sqrt{|n|}+\frac{-1}{2}\sqrt{(2t+1)}$$ then gives
 %\begin{equation*}\begin{aligned}
\begin{multline}
\sum_{n<0} b_g(n)(-4\pi n)^{1-k}\int_0^{\infty}\frac{(\scrL |(\varphi_s^w|_{2-k}W_N)_{2-k}|)(-2\pi n(2t+1))}{(1+t)^k} dt \\
\ll
\sum_{n<0} |b_g(n)||n|^{\frac{2 \sigma+1-4k}{4}}e^{-\sqrt{\frac{2 \pi x|n|}{N}}}\int_0^{\infty}\frac{e^{-\sqrt{\frac{2 \pi x(2t+1)}{N}}}}{(1+t)^{k-\frac{\sigma}{2}+\frac{3}{4}}}dt
% \end{aligned}\end{equation*}
\end{multline}
By \eqref{coeffbound} this converges for $x >C^2_gN/(2 \pi)$, which proves that $\varphi_s^w|_{2-k}W_N \in \mathcal F_g$.
\end{proof}

This lemma allows us to apply Theorem 4.5 of \cite{DLRR} to deduce our functional equation.
\begin{proposition}\label{FE} Fix $k\in \Z$ and $N\in \N$.
Suppose that $f$ is an element of $HC_k(N)$ and that $g=f|_kW_{N}$ with expansions \eqref{FourEx}. Further suppose that $C_f, C_g$ are positive constants such that  \eqref{coeffbound} holds. Then, for each $w$ with $\Re (w)>\max (2 \pi n_0, NC_g^2/{(2\pi)})$
and for each $s \in \mathbb C$, we have
$$L_f(\varphi_s^{w})=i^k N^{1-\frac{k}{2}}L_g(\varphi_{s}^{w}|_{2-k}W_N).$$
%where $(F \circ N)(x):=F(Nx)$ for each map $F: \mathbb R \to \mathbb C.$
\end{proposition}

We now prove a lemma that extends the crucial identity employed in the proof of Th. 3.2 of \cite{BFI}. Here and in the sequel, the implied branch of the logarithm is the principal one.
\begin{lemma}\label{bend}
For each $w \in \HH,$ we have
\begin{equation}\label{bendeq}
i^{a} E_{1-a}(w)=
\int_i^{i+\infty} e^{iwz} z^{a-1} dz.
\end{equation}
for all $a \in\mathbb R.$  If $\Im(w)=0$ and $\Re(w)>0$, then \eqref{bendeq} holds for all $a<0.$
\end{lemma}
\begin{proof}
Let $w \in \C$ be such that $\Im(w) \ge 0.$ We first assume that $\Re(w)>0.$ Then, for $T>1$,
$$i^a\int_1^{T} e^{-w t} t^{a-1} dt=
\int_i^{iT} e^{iwz} z^{a-1} dz=
\left(\int_i^{i+T}+\int_{i+T}^{iT+T}+\int_{iT+T}^{iT}\right)e^{iwz} z^{a-1}dz.
$$
We consider the last two integrals in turn. We first bound the third integral:
$$\int_{iT+T}^{iT} e^{i w z} z^{a-1} dz=
\int_T^0 e^{iw (iT+t)}(iT+t)^{a-1} dt
%$$ This is bounded by
%$$e^{-\Re(w)T}\int_0^T e^{-t\Im(w)}|iT+t|^{a-1} dt
\ll T^a e^{-\Re(w)T} \int_0^1 |i+t|^{a-1} dt.
$$
This vanishes as $T \to \infty.$
For the second integral we have
\begin{equation}\label{bound1}\int_{i+T}^{iT+T} e^{iw z} z^{a-1} dz=
ie^{iw T}\int_1^T e^{-w t}(T+it)^{a-1} dt
\ll e^{-\Im(w)T} \int_1^{T}|T+it|^{a-1}dt
%+T^{a+1}\int_1^{\infty}e^{-\frac{\Re(w_2)t+\Im(w_2) }{NT(1+t^2)}}|1+it|^{a-1}dt \right )
\end{equation}
If $\Im(w)>0$, this vanishes as $T \to \infty,$ for all $a \in \mathbb R.$
If $\Im(w)=0$ and $a<0$, 
%we can apply
%the Cauchy--Schwarz inequality to bound this integral from above by $$2^{\frac{1-a}{2}} e^{-\Im(w)T} \int_1^{T}(T+t)^{a-1}dt\le 2^{\frac{1-a}{2}} \int_1^{\infty}(T+t)^{a-1}dt$$
it is bounded from above by $\int_1^{\infty}(T^2+1)^{(a-1)/2}dt$
which also vanishes as $T \to \infty.$

 Therefore, under the conditions of the lemma and, if $\Re(w)>0$, we have
\begin{equation*}
\begin{aligned} i^{a}E_{1-a}(w)  & = i^{a}\int_1^{\infty} e^{-wt} t^{a-1} dt \\
& =\lim_{T \to \infty} \left (\int_i^{i+T}+\int_{i+T}^{iT+T}+\int_{iT+T}^{iT}
\right ) e^{iw z} z^{a-1} dz
=
\int_i^{i+\infty} e^{iwz} z^{a-1} dz
+0+0.
\end{aligned}
\end{equation*}
By the remarks on the analytic continuation of $E_{1-a}(w)$ in the beginning of the section, we see that both sides of \eqref{bendeq} are analytic in $\{w \in \mathbb C; \Im(w)>0\}$. Therefore, by the identity theorem, we deduce that \eqref{bendeq} remains true for all $w$ with $\Im(w)>0.$
\end{proof}

We can now prove our main formula for $L_f(\varphi_s^w)$. We
assume that $w \in \HH$ and, in the first instance, that
$\Re(w)>\max(2 \pi n_0, C_g^2N/(2\pi))$.
By Lemma \ref{Ffg} ,
%and Lemma 4.4 of \cite{DLRR},
$L_f(\varphi_s^w)$ is well-defined for all $s \in \mathbb R$.
% and, by Lemma \ref{lem:LfLf'_int}, it has an integral expression.   
Using \eqref{nonhol}, we split up the $L$-value as
%$$L_f(\varphi_j^w)= \int_0^{\infty}f(it)\varphi_j^w(t)dt.$$ Suppose that $j<0.$ Then
\begin{equation}\label{series}
L_f(\varphi_s^w)=I_1+I_2,
\end{equation}
where
\begin{equation*}
\begin{aligned}
I_1&:=\sum_{\substack{n \ge -n_0 \\ n \ne 0}} a_f(n) (\scrL\varphi_s^w)\left (2 \pi n\right )=
\sum_{\substack{n \ge -n_0 \\ n \ne 0}} a_f(n) \int_1^{\infty} e^{-(2\pi n+w)t} t^{s-1}dt
%E_{1-j}\left (2 \pi n+x\right )
\end{aligned}
\end{equation*}
and
\begin{equation*}
\begin{aligned}
I_2:&=\sum_{n<0} b_f(n)\int_0^{\infty}\Gamma\left(1-k, -4 \pi n t\right)e^{-2 \pi n t}\ph_s^w(t)dt\\
&=
\sum_{n<0} b_f(n)\int_1^{\infty}\Gamma\left(1-k, -4 \pi n t\right)e^{-(2 \pi n+w)t}t^s \frac{dt}{t}.
\end{aligned}
\end{equation*}
By Lemma \ref{bend}, applied with $2 \pi n+w$ and $s$ in place of $w$ and $a$ respectively, we deduce that
\begin{equation*}
I_1=i^{-s}\sum_{\substack{n \ge -n_0 \\ n \ne 0}} a_f(n) \int_i^{i+\infty} e^{i(2\pi n+w)z} z^{s-1}dz.
\end{equation*}
We now turn to $I_2$. For each $n<0$, it is easy to see (first for $\Re(w)>-2 \pi n$ and, by analytic continuation for all $w \in \mathbb H^+_0$), that
\begin{equation*}
\begin{aligned}&\int_1^{\infty}\Gamma\left(1-k, -4 \pi n t\right)e^{-(2 \pi n+w)t}t^s \frac{dt}{t}\\&\qquad \qquad =
\Gamma(1-k, -4 \pi n)E_{1-s}(2 \pi n +w)
-(-4 \pi n)^{1-k}
\int_1^{\infty}e^{4 \pi n t}t^{s-k}E_{1-s}((2 \pi n+w)t)dt.
\end{aligned}
\end{equation*}
With Lemma \ref{bend}, this becomes
$$i^{-s}\Gamma(1-k, -4 \pi n)\int_i^{i+\infty}e^{i(2 \pi n +w)z} z^s\frac{dz}{z}
-(-4 \pi n)^{1-k}
\int_1^{\infty}e^{4 \pi n t}t^{s-k}E_{1-s}((2 \pi n+w)t)dt.
$$
Since $w \in \HH^+,$ the bounds \eqref{coeffbound}, \eqref{asym} and \eqref{boundE} imply that we can interchange summation and integration in $I_1$ and $I_2$ to deduce that $L_f(\varphi_s^w)$ equals
\begin{equation}\label{gen}
%L(\varphi_s^w)=
i^{-s}\int_i^{i+\infty} f(z)e^{i w z} z^{s-1}dz
-\sum_{n<0} b_f(n)(-4 \pi n)^{1-k}
\int_1^{\infty}e^{4 \pi n t}t^{s-k}E_{1-s}((2 \pi n+w)t)dt.
\end{equation}
Here and in the sequel, the path of integration of the first integral is fixed to be $\{i+t; t \in \mathbb R_{>0}\}$ so that $\Im(z)=1.$

The first term of the right hand side of \eqref{gen} can be expressed as an integral over a finite interval as follows:
\begin{equation}\label{lerch}
i^{-s}\sum_{m=0}^{\infty} \int_{i+m}^{i+m+1} f(z)e^{i w z} z^{s-1}dz=
i^{-s}\int_{i}^{i+1} f(z)e^{i w z} \zeta \left (1-s, \frac{w}{2\pi}, z \right ) dz.
\end{equation}
 The values of the parameters that appear in \eqref{gen} belong to the domain of absolute convergence of the defining series of $\zeta(s, a, c)$. 
%$$\psi^{(-j)}(z):=\sum_{m=0}^{\infty} (z+m)^{j-1}=\frac{d^{-j+1}}{dz^{-j+1}}\log \Gamma(z)$$ is the polygamma function. 

The second term of the right hand side of \eqref{gen} does not appear if $f$ is weakly holomorphic, but, when it does, it can be explicitly computed for integer values of $s.$ This will be done in the next section. However, it can also be written as an integral over the same finite interval in terms of $\xi_k f$, to give a more unified appearance to the general formula of $L_f(\varphi_s^w).$ Specifically, as above, we can apply Lemma~\ref{bend}
to $E_{1-s}((2 \pi n+w)t)$ and, since, for $w \in \HH^+_0,$ it is legitimate to interchange the order of integration, we obtain
$$
%\begin{multline}
\int_1^{\infty}e^{4 \pi n t}t^{s-k}E_{1-s}((2 \pi n+w)t)dt=
i^{-s}\int_i^{i+\infty} \int_1^{\infty}e^{i(2 \pi n +w)t z+4 \pi n t} z^{s-1}t^{s-k} dt dz.
%\end{multline}
$$
Applying \eqref{xi}, we deduce that
\begin{multline}\label{2ndterm}
R(w, s):=-\sum_{n<0} b_f(n)(-4 \pi n)^{1-k}
\int_1^{\infty}e^{4 \pi n t}t^{s-k}E_{1-s}((2 \pi n+w)t)dt
\\=
i^{-s}\int_i^{i+\infty} \int_1^{\infty}(\xi_k f^c)(t(2i-z))e^{itzw} t^{s-k} z^{s-1} dt dz=
i^{-s}\int_i^{i+1} \int_1^{\infty} e^{itzw}t^{s-k}R_t(z, w)dtdz,
\end{multline}
where 
\begin{equation}\label{R_t}R_t(z, w):=\sum_{m=0}^{\infty}\frac{(\xi_k f^c)(t(2i-z-m))}{(z+m)^{1-s}}e^{itmw}.
\end{equation}
It is routine to check that interchanges of summation and integrations are justified once we recall that $\xi_k f^c$ is a holomorphic cusp form. 

We now note that \eqref{gen} defines an analytic function in $\mathbb{H}^+$ with a continuous extension to $\mathbb{H}^+_0.$ Indeed, the Lerch zeta function is analytic in the region it is evaluated in and the second term of \eqref{gen} is shown in \eqref{2ndterm} to equal an analytic function of $w \in \mathbb H^+$ with a continuous extension to $\mathbb{H}^+_0$. Further, for fixed $s$, $L_f(\varphi_s^w)$ can be analytically continued, as a function of $w$, to an open set containing $\mathbb H^{+}_0$: By \eqref{series}, we have that, for all $w \in \HH$ with $\Re(w)>\max(2 \pi n_0, C_g^2N/(2\pi))$, we have
\begin{equation}\label{expint}
L_f(\varphi_s^w)=\sum_{\substack{n \ge -n_0 \\ n \ne 0}} a_f(n) E_{1-s}\left (2 \pi n+w\right ) +I_2.
\end{equation}
The first series in the right hand side of \eqref{expint} converges in compacta for $w\in\mathbb H$, because each term is well-defined and, by \eqref{coeffbound} and \eqref{boundE},
%as $n \to \infty$, 
$$a_f(n) E_{1-s}\left (2 \pi n+w\right ) \ll e^{C_f \sqrt{n}-2 \pi n} \qquad \text{as $n \to \infty$.} 
$$
Further, by two applications of \eqref{asym}, we see that, for $w$ in a compact subset of $\mathbb H^+_0$, we have that, for some $K>0,$
$$\int_0^{\infty}\Gamma\left(1-k, -4 \pi n t\right)e^{-2 \pi n t}\ph_s^w(t)dt \ll e^{K n} \qquad \text{as $n \to -\infty$.} 
$$
Therefore, both series in \eqref{expint} can be analytically continued to $\mathbb H^+$ with a continuous extension to $\mathbb H^+_0$. 

Since, as we proved above, $L_f(\varphi_s^w)$ equals  \eqref{gen} when $\Re(w)>\max(2 \pi n_0, C_g^2N/(2\pi))$ and $\Im(w)>0$, by analytic continuation and continuity we have shown our main result:
\begin{theorem}\label{main} Let $f \in HC_k(N)$. Then, for each $s \in \mathbb R$ and each $w \in \mathbb H^+_0$,
we have
\begin{equation*}
L_f(\varphi_s^w)= 
i^{-s}\int_{i}^{i+1} \left ( f(z)e^{i w z} \zeta \left (1-s, \frac{w}{2\pi}, z \right )+\int_1^{\infty} e^{itzw}t^{s-k}R_t(z, w)dt \right )dz
\end{equation*}
where the path of integration of the outer integral is fixed to be $\{i+t; t \in [0, 1]\}$
Here $\varphi_s^w(t)=\mathbf 1_{[1, \infty)}(t) e^{-wt}t^{s-1}$,
$\zeta(s, a, c)$ is the Lerch zeta function and $R_t$ is defined by \eqref{R_t}.
\end{theorem}
\begin{rmk}
As $S_k^!(N)\subset HC_k(N)$ and the restriction to $w \in \HH^+_0$ is required only for the treatment of the ``non-holomorphic part" of $f$, Theorem~\ref{main} includes Theorem~\ref{maincor} as a special case.
\end{rmk}

\section{The values of $L^*(f, s)$}\label{appl}
In this section, we will apply the general set up and main theorem we proved in the last section to evaluate values at specific $s \in \mathbb R$. We will treat separately the following cases: $s=0$ (which corresponds to the setting considered in \cite{BFI}), $s=1+m$, with $m \in \mathbb N_0$ and $s<0$ (not necessarily integral). Those cases together imply Th. \ref{mainInt*}.
These results could be extended to all harmonic Maass forms but we will do so only in the case of $s=1+m$ with $m \in \mathbb N_0$, because the technicalities considerably increase in the other cases, due to the presence of infinite sums, without adding any substantially new idea.

\subsection{The case $s=0$}\label{s=0} We first retrieve the ``$L$-value" and its formula discussed in \cite{BFI}. As mentioned in the introduction, given a weakly holomorphic modular form $g$ , the authors of \cite{BFI} constructed a series of the form \eqref{reg} which is crucial for their main result and which can be thought off as a ``central $L$-value" of $g$, although no $L$-series is defined there.
Within our framework, this now becomes an actual $L$-value.

Specifically, as mentioned above, if $f$ is a weakly holomorphic cusp form, then for each $s \in \mathbb R,$ $L_f(\varphi_s^w)$ has an analytic continuation to $\mathbb H$ and it is continuously extendable to $\mathbb R-\{2 \pi n; n=\pm 1, \pm 2, \dots \}$. 
%as $w$ approaches it from the upper half-plane
This implies that, for each $s \in \mathbb R,$ $L_f(\varphi_s^0)$ is well-defined and, for $s=0$, it equals the $L$-value $L^*(f, 0)$ as defined in \eqref{L(f}.

On the other hand, Th. \ref{main} implies that, for $x>0$,
$$L_f(\varphi_0^{ix})=
\int_{i}^{i+1} f(z)e^{-x z} \zeta \left (1, \frac{ix}{2\pi}, z \right ) dz.
$$
Since $f$ has a vanishing constant term in its Fourier expansion, we have
\begin{equation}\label{transl}
\int_i^{i+1}f(z)dz=0.
\end{equation}
Therefore,
\begin{equation*}
    \begin{aligned}
    \int_{i}^{i+1} f(z)e^{-x z} \zeta \left (1, \frac{ix}{2\pi}, z \right ) dz=
   \int_{i}^{i+1}  & f(z)e^{-x z} \left ( \zeta \left (1, \frac{ix}{2\pi}, z \right )-\zeta \left (1, \frac{ix}{2\pi}, 1 \right ) \right ) dz\\+
    &\int_{i}^{i+1} f(z) (e^{-x z}-1) \zeta \left (1, \frac{ix}{2\pi}, 1 \right ) dz.
\end{aligned}
\end{equation*}
We now observe that $\zeta \left (1, \frac{ix}{2\pi}, 1 \right )=-e^{x} \log(1-e^{-x}),$ and hence
\begin{equation}\label{lim0}
(e^{-x z}-1) \zeta \left (1, \frac{ix}{2\pi}, 1 \right )=
\frac{e^{-x z}-1}{e^{-x }-1}e^{x}(1-e^{-x }) \log(1-e^{-x}) \to 0 \quad (\text{as $x \to 0^+$).}
\end{equation}
We then expand
$$\zeta \left (1, \frac{ix}{2\pi}, z \right )-\zeta \left (1, \frac{ix}{2\pi}, 1 \right )=(1-z)\sum_{m=0}^{\infty}\frac{e^{-xm}}{(z+m)(1+m)},$$
which converges uniformly to $\sum_{m=0}^{\infty}((z+m)^{-1}-(1+m)^{-1})$. This and \eqref{lim0}, together with \eqref{transl} imply that
\begin{equation}
\lim_{x \to 0^+} \int_{i}^{i+1} f(z)e^{-x z} \zeta \left (1, \frac{ix}{2\pi}, z \right ) dz=-\int_{i}^{i+1} f(z) \psi(z)dz,
\end{equation}
where we used that the digamma function $\psi(w)$ can be expanded as
$$\psi(w)=\frac{\Gamma'(w)}{\Gamma(w)}=-\gamma+\sum_{m=0}^{\infty}
\left (\frac{1}{m+1}-\frac{1}{m+w}\right).$$
This implies the following proposition which, in the case of the $j$-invariant, is a version of a formula suggested by Zagier generalized to all weight $0$ weakly holomorphic modular forms in \cite{BFI}.
\begin{proposition}\label{Zag}
Let $k \in \mathbb Z$ and let $f \in S^!_k(N)$. Then,
\begin{equation}\label{identZag}
L_f(\varphi_0^0)=-\int_{i}^{i+1} f(z) \psi(z)dz.
\end{equation}
\end{proposition}
\begin{rmk}\label{cconj}
The difference of the appearance of \eqref{identZag} from that of \eqref{ZagInt} (and more generally, of Th. 1.1 of \cite{BFI}) is
explained by \eqref{EJ-E_1}. The identity \eqref{ZagInt} involves $\mathcal {EI}(2 \pi n)$, for $n<0$, whereas ours uses $E_1(2 \pi n)$ whose real part, by \eqref{EJ-E_1}, is $\mathcal {EI}(2 \pi n).$ Since, further, $\int_i^{i+1} j(z) \psi(1-z)dz$ is the complex conjugate of
$\int_i^{i+1} j(z) \psi(z)dz$, we see that the two formulas are consistent.
\end{rmk}

\subsection{The case $s=1+m$ (with $m \in \mathbb N.$)}\label{s=1+m}
In this case, we will work with a general harmonic Maass cusp form.

We first compute the ``principal part" and constant term in the expansion in $e^{-x}$ of $e^{-x z} \zeta(-m, ix/(2 \pi), z)$ around $1.$ For $0<y<\epsilon$ and $z \in \HH^+,$ set
$$\phi(-m, y, z):=\sum_{n \ge 0}y^{n+z}(n+z)^m.$$ Then
$e^{-x z} \zeta(-m, ix/(2 \pi), z)=\phi(-m, e^{-x}, z)$ and
\begin{equation}\label{recur}y\frac{d}{dz}\phi(-m, y, z)=\phi(-m-1, y, z).\end{equation} Then we have
\begin{lemma} \label{indep} For each $m \in \N_0$, $z \in \HH^+$ and $0<y<\epsilon$, there are $a_j$ ($1 \le j \le m+1$) and polynomials $b_j(z)$ ($j \in \N_0$) such that 
$$\phi(-m, y, z)=\sum_{j=1}^{m+1}\frac{a_j}{(1-y)^j}+\sum_{j \ge 0}b_j(z)(1-y)^j.
$$
\end{lemma}
\begin{proof} We first see that, with the binomial series, there are polynomials $b_j(z)$ such that,
$$\phi(0, y, z)=\frac{y^z}{1-y}=\frac{1}{1-y}-z+\sum_{j \ge 1} b_j(z) (1-y)^j.$$
Assume that the statement holds for $m \ge 0,$ i.e.
$$\phi(-m, y, z)=\sum_{j=1}^{m+1}\frac{a_j}{(1-y)^j}+\sum_{j \ge 0} b_j(z)(1-y)^j.$$ Then, with \eqref{recur}, we have
$$\phi(-m-1, y, z)=\frac{(m+1)a_{m+1}}{(1-y)^{m+2}}+\sum_{j=1}^{m+1}\frac{(j-1)a_{j-1}-ja_j}{(1-y)^j}+\sum_{j \ge 1} jb_j(z) ((1-y)^j-(1-y)^{j-1}).$$
By induction, we deduce the lemma.
\end{proof}
Next, we will prove the following:
\begin{lemma} \label{b_0} With the notation of Lemma \ref{indep}, we have
$$b_0(z)=-\frac{B_{m+1}(z)-B_{m+1}}{m+1}$$
where $B_{m+1}(z)$ is the standard Bernoulli polynomial
and $B_{m+1}=B_{m+1}(0)$ the Bernoulli number.
\end{lemma}
\begin{proof}
We will use Th. 6 of \cite{KKY}, which asserts that
$$\phi(-m, y, z)=\sum_{r=0}^m\sum_{h=r}^m\binom{m}{h}S(h, r)z^{m-h} \frac{r!y^{r+z}}{(1-y)^{1+r}}$$
where $S(h, r)$ is the Stirling number of the second kind. With the binomial series applied to $y^{r+z}=(1-(1-y))^{r+z}$, we deduce
$$\phi(-m, y, z)=\sum_{\ell \ge 0}\sum_{r=0}^m\sum_{h=r}^m\binom{m}{h}S(h, r)z^{m-h} \frac{r!(-1)^{\ell}}{\ell! (1-y)^{1+r-\ell}}(r+z-\ell+1)_{l}$$
where $(a)_n$ stands for the Pochhammer symbol. This implies that $b_0(z)$ equals
$$
%b_0(z)=
\sum_{r=0}^m\frac{(-1)^{r+1}}{r+1}\left (\sum_{h=r}^m\binom{m}{h}S(h, r)z^{m-h} \right ) (z)_{r+1}=\sum_{h=0}^{m} \binom{m}{h}z^{m-h}\sum_{r=0}^h\frac{(-1)^{r+1}}{r+1}(z)_{r+1}S(h, r).$$
The identity (see, e.g. (26.8.35) and (24.4.7) of \cite{NIST})
$$B_{n+1}(x)=B_{n+1}+\sum_{k=0}^n\frac{(-1)^{k+1}(n+1)}{k+1}S(n, k)(-x)_{k+1}$$ implies that the inner sum equals
$( B_{h+1}(-z)-B_{h+1})/(h+1).$ Therefore,
\begin{multline*}
b_0(z)=\frac{1}{m+1}\sum_{h=0}^m \binom{m+1}{h+1}z^{m-h}(B_{h+1}(-z)-B_{h+1}(0))\\
=
\frac{1}{m+1}\sum_{h=0}^{m+1}
\binom{m+1}{h}z^{m-h+1}(B_{h}(-z)-B_{h}(0))
\end{multline*}
By the translation property of Bernoulli polynomials ((24.4.12) of \cite{NIST}), we deduce that this is equal to
$$\frac{1}{m+1} \left ( B_{m+1}(0)-B_{m+1}(z)\right ).$$
\end{proof}

Theorem \ref{main}, combined with the last two lemmas, implies
\begin{equation}
    \begin{aligned}
\label{Lm} L_f(&\varphi_{1+m}^{ix} )=i^{-m-1} \int_{i}^{i+1} f(z)e^{-x z} \zeta \left (-m, \frac{ix}{2\pi}, z \right ) dz+R(ix, 1+m)\\&=
i^{-m-1}\left (\frac{B_{m+1}}{m+1}+\sum_{j=1}^{m+1}\frac{a_j}{(1-e^{-x})^j}\right )\int_{i}^{i+1} f(z)dz \\
& \qquad +i^{-m-1}\int_{i}^{i+1} f(z) \left (\frac{-B_{m+1}(z)}{m+1}+\sum_{j \ge 1}b_j(z) (1-e^{-x})^j \right )dz+R(ix, 1+m)
\end{aligned}
\end{equation}
where we recall that $R$ is defined in \eqref{2ndterm}.
The first integral on the right hand side of \eqref{Lm} is $0$, because of \eqref{transl}.
Therefore, the limit of the first two terms as $x \to 0^+$ is \begin{equation}\label{1stpiece}-i^{-m-1}\int_{i}^{i+1} f(z) \frac{B_{m+1}(z)}{m+1}dz.\end{equation}

To deal with $R(ix, 1+m)$, we first note that this is present only if $k \le 0$, therefore, for the remainder of the proof of Th. \ref{bern}, we will assume so. With (8.4.8) of \cite{NIST},
$$E_{-m}((2 \pi n+ix)t)=m!((2 \pi n+ix)t)^{-1-m}e^{-(2 \pi n+ix)t}\sum_{\ell=0}^{m} \frac{((2 \pi n+ix)t)^\ell}{\ell!}.$$
Therefore, with \eqref{E-G},
\begin{equation}\label{integralR}\int_1^{\infty}e^{4 \pi n t}t^{1+m-k}E_{-m}((2 \pi n+ix)t)dt=
m!\sum_{\ell=0}^{m} \frac{(2 \pi n+ix)^{\ell-m-1}}{\ell!}
E_{k-\ell}(ix-2 \pi n).\end{equation}
Equation (8.4.8) of \cite{NIST} implies, for all $z \in \mathbb H$
$$E_{k-\ell}(z)=\frac{(\ell-k)!}{(1+m-k)!}z^{1+m-\ell}E_{k-1-m}(z)-e^{-z}\sum_{j=0}^{m-\ell}z^{j}\frac{(\ell-k)!}{(j+\ell-k+1)!}$$
and, thus, \eqref{integralR} becomes
\begin{multline}\label{integralR2}
m!\sum_{\ell=0}^{m} \frac{(2 \pi n+ix)^{\ell-m-1}}{\ell!}
\frac{(\ell-k)!}{(1+m-k)!} (ix-2 \pi n)^{1+m-\ell}E_{k-1-m}(ix-2 \pi n) \\
-e^{2 \pi n-ix}m!\sum_{\ell=0}^{m}
\frac{(2 \pi n+ix)^{\ell-m-1}}{\ell!} \sum_{j=0}^{m-\ell}
(ix-2 \pi n)^{j} \frac{(\ell-k)!}{(\ell-k+j+1)!}.\end{multline}
We first determine the contribution of the first term of \eqref{integralR2} to the final formula. Using Lemma \ref{bend}, we get
\begin{multline}\label{integralR3}
-\sum_{n<0} b_f(n)(-4 \pi n)^{1-k}\sum_{\ell=0}^{m} \frac{(2 \pi n+ix)^{\ell-m-1}}{\ell!}
\frac{m!(\ell-k)!}{(1+m-k)!} (ix-2 \pi n)^{1+m-\ell}E_{k-1-m}(ix-2 \pi n) \\
=-i^{k-m}\sum_{\ell=0}^{m} \frac{m!(\ell-k)!}{(1+m-k)! \ell!} \int_i^{i+\infty} (\xi_k f^c)(\ell, ix; z) e^{-xz}z^{1+m-k}dz\end{multline}
where
\begin{equation}\label{xiw}
(\xi_kf^c)(\ell, ix; z):=-\sum_{n<0} (-4 \pi n)^{1-k} b_f(n) \left ( \frac{2 \pi n+ix}{-2 \pi n+ix}\right )^{\ell-1-m}e^{-2 \pi i n z}
\end{equation}
is a twisted version of $\xi_k f^c$. It is easy to see that this series is absolutely and uniformly convergent in terms of both $x$ and $z$ and that its limit as $x \to 0^+$ exists and equals $(-1)^{\ell-m-1} (\xi_k f^c).$

As in the proof of Th. \ref{main} and thanks to the periodicity of $(\xi_kf^c)(\ell, ix; z)$ in $z$, we deduce
$$\int_i^{i+\infty} (\xi_kf^c)(\ell, ix; z) e^{-xz}z^{1+m-k}dz=
\int_i^{i+1} (\xi_kf^c)(\ell, ix; z) \zeta(k-1-m, \frac{ix}{2 \pi}, z) e^{-xz}dz.$$ With Lemmas \ref{indep} and \ref{b_0}, we see that
$e^{-xz} \zeta(k-1-m, ix/(2 \pi), z)$ equals
$$\sum_{j=1}^{m+2-k}
\frac{a_j}{(1-e^{-x})^j}+ \frac{B_{m+2-k}}{m+2-k}-\frac{B_{m+2-k}(z)}{m+2-k}+\sum_{j \ge 1} b_j(z) (1-e^{-x})^j$$
for some constants $a_j$ and polynomials $b_j(z)$. 
Since the constant term in the Fourier series $(\xi_kf)(\ell, ix; z)$ vanishes, we have $\int_i^{i+1}(\xi_kf)(\ell, ix; z)dz=0$ for all $x>0.$
Therefore, the limit of \eqref{integralR3} as $x \to 0^+$ is
\begin{equation}\label{2ndpiece}
c_{k,m}\int_i^{i+1}(\xi_kf^c)(z) \frac{B_{2+m-k}(z)}{m+2-k}dz,\end{equation}
where
$$c_{k,m}:=-\sum_{\ell=0}^{m} \frac{m!(\ell-k)! i^{k+m+2\ell} }{(1+m-k)! \ell!}.$$

For the second term of \eqref{integralR2}, we first observe that, because of the exponential decay of $e^{2 \pi n}$, as $n \to -\infty$ and the polynomial growth of $(-4 \pi n)^{1-k}$ and $(2 \pi n)^\ell,$ the series
$$\sum_{n<0} b_f(n) (-4 \pi n)^{1-k} e^{2 \pi n-ix}m!\sum_{\ell=0}^{m}
\frac{(2 \pi n+ix)^{\ell-m-1}}{\ell!} \sum_{j=0}^{m-\ell}
(ix-2 \pi n)^{j} \frac{(\ell-k)!}{(\ell-k+j+1)!}.$$
converges absolutely and uniformly in terms of $x$ and $z$, and the limit as $x \to 0^+$ equals
$$m!\sum_{\ell=0}^{m}\sum_{j=0}^{m-\ell} (-1)^{\ell-m-1}2^{1-j-\ell+m} \frac{(\ell-k)!}{\ell! (\ell-k+j+1)!}
\sum_{n<0} b_f(n) (-4 \pi n)^{\ell-m-k+j} e^{2 \pi n}.$$
Because of the expansion
$B_r(x)=-r!(2 \pi i)^{-r}\sum_{j \ne 0} e^{2 \pi i j x} j^{-r}$ (e.g. (24.8.3) of \cite{NIST}), the last expression becomes
\begin{multline}\label{remaind}\sum_{\ell=0}^{m}\sum_{j=0}^{m-\ell}  d_{\ell, j}
\sum_{n<0} b_f(n) (-4 \pi n)^{1-k} e^{2 \pi n}\int_0^1 e^{-2 \pi inx}B_{1-\ell+m-j}(x)dz \\
=-\sum_{\ell=0}^{m}\sum_{j=0}^{m-\ell}  d_{\ell, j}
\int_i^{i+1} (\xi_k f^c)(z) B_{1-\ell+m-j}(x)dz \end{multline}
where $x$ here stands for the real part of $z$ and
$$d_{\ell, j}:=m!  \frac{(-1)^{j-1} (\ell-k)!}{\ell! (j+\ell-k+1)! (1-\ell+m-j)!}.$$

Therefore, upon taking the limit as $x \to 0^+$ in the right hand side of \eqref{Lm} and combining \eqref{1stpiece}, \eqref{2ndpiece} and \eqref{remaind}, we deduce
\begin{theorem}\label{bern} Let $f \in HC_k(N).$ For each $m \ge 1,$ we have
\begin{multline}
L_f(\varphi_{1+m}^{0})=
-i^{-m-1}\int_{i}^{i+1} f(z) \frac{B_{m+1}(z)}{m+1}dz \\ +
\int_i^{i+1}(\xi_kf^c)(z) \left ( c_{k,m} \frac{B_{2+m-k}(z)}{2+m-k}-
\sum_{\ell=0}^{m}\sum_{j=0}^{m-\ell} d_{\ell, j} B_{1-\ell+m-j}(x) \right ) dz
\end{multline}
where $x$ denotes the real part of $z$.
\end{theorem}
\begin{remark} When comparing with the results of \cite{AS}, one should keep in mind that they work with forms of positive weight
and a slightly different version of harmonic Maass forms (those called \emph{of moderate growth} in \cite{HMFBook}). The harmonic Maass forms we are studying here are weakly holomorphic when the weight is positive. Therefore, in the non-holomorphic case, our formulas cannot be compared directly with theirs. However, their respective structure is entirely compatible.
\end{remark}
In the case of a weakly holomorphic cusp form $f$, the formula simplifies to give
\begin{corollary}\label{bernWHF} Let $f \in S_k^!(N).$ For each $m \ge 1,$ we have
$$L_f(\varphi_{1+m}^{0})=
-i^{-m-1}\int_{i}^{i+1} f(z) \frac{B_{m+1}(z)}{m+1}dz.$$
\end{corollary}
\begin{remark}
By the well-known formula $\zeta(-m, a)=-B_{m+1}(a)/(m+1),$ for the Hurwitz zeta function, we notice that the formula of this theorem is the same one we would have arrived at, if we directly set $w=0$ in the formula of Th. \ref{main}. However, that would not be a valid proof because it would require the continuity of $\zeta \left (-m, \frac{ix}{2 \pi}, z \right )$ as $x \to 0^+$, which does not hold.
%(as can be seen by \eqref{Erd}).
\end{remark}

\subsection{The case $s=1$}\label{s=1}
\begin{corollary}\label{polyl} Let $f \in HC_k(N)$. Then, we have
\begin{equation*}
L_f(\varphi_{1}^0)=
i\int_{i}^{i+1} f(z)z dz-\frac{1}{1-k}
\int_i^{i+1}(\xi_kf^c)(z) \left (i^k \frac{B_{2-k}(z)}{2-k}+
x \right ) dz
\end{equation*}
where $x$ denotes the real part of $z.$
\end{corollary}
\begin{proof} As above, the $L_f(\varphi_s^{ix})$ is well-defined as $x \to 0^+$. 
We next note that for all $m \in \mathbb N_0$, we have
$$\zeta \left (-m, \frac{w}{2 \pi}, z \right )=e^{-iw}Li_{-m}(e^{iw}, z),$$
where $$Li_{n}(a, z)=\sum_{k=0}^{\infty}\frac{a^{k+1}}{(z+k)^n}$$ is the \emph{$z$-deformed negative polylogarithm} (\cite{LL}).
Since 
$Li_{0}(e^{iw}, z)=e^{iw}(1-e^{iw})^{-1}$, Prop. \ref{main} implies that, for $x>0$
$$L_f(\varphi_{1}^{ix})=i^{-1}\int_i^{i+1} f(z)e^{-xz}\frac{1}{1-e^{-x}}dz+R(ix, 1)=
i^{-1}\left ( \int_i^{i+1} f(z)\frac{e^{-xz}-1}{1-e^{-x}}dz \right )+R(ix, 1),
$$
where, for the last equality we used \eqref{transl}.
The limit of the fraction inside the last integral is $-z$ and thus
\begin{equation}\label{m=0}
L_f(\varphi_{1}^0)=i\int_{i}^{i+1} f(z)z dz+\lim_{x \to 0^+}R(ix, 1).
\end{equation}
We now observe that all arguments used in the proof of Th. \ref{bern} to compute the limit of $R(ix, 1+m)$ are valid when $m=0$. Therefore, by specializing the outcome of that computation to $m=0$, and recalling that, if this part is present, then $k \le 0$, we deduce
$$\lim_{x \to 0^+}R(ix, 1)=-\frac{1}{1-k}
\int_i^{i+1}(\xi_kf^c)(z) \left (i^k \frac{B_{2-k}(z)}{2-k}+
x-\frac{1}{2} \right ) dz.$$
This, together with \eqref{transl} and \eqref{m=0} implies the corollary.
\end{proof}
\begin{rmk}
Cor. \ref{polyl} can be written in the same way as Cor. \ref{hurw}, because of the identity $\zeta(0, a)=\frac{1}{2}-a$ and \eqref{transl}.
\end{rmk}

\subsection{The case $s<0$}\label{s<0}
In this case, we can even express non-integer values in the form of Prop. \ref{Zag}:
\begin{corollary} \label{hurw}Let $f \in S^!_k(N)$. Then, for each $s<0$, we have
\begin{equation}\label{hurweq}
L_f(\varphi_s^0)=
i^{-s}\int_{i}^{i+1} f(z) \zeta(1-s, z)
dz.
\end{equation}
\end{corollary}
\begin{proof} As mentioned above, $L_f(\varphi_s^0)$ is well-defined as the limit of $L_f(\varphi_s^{ix})$, as $x \to 0^+.$
On the other hand, for $s>0,$ the series defining $\zeta(1-s, ix/(2 \pi), z)$ converges absolutely and uniformly to the Hurwitz zeta function, as $x \to 0^+$. %%
From Th. \ref{main}, we deduce the result, by taking the limit as $w=ix \to 0$ (with $x>0$).
\end{proof}
\begin{remark}\label{polyG}
When $s$ is a (negative) integer, we have
$(-s)!\zeta(1-s, 0, z)=(-1)^{1-s}\psi^{(-s)}(z)$, where $\psi^{(m)}(w)$ denotes the polygamma function
$$\psi^{(m)}(w):=\frac{d^{m+1}}{dz^{m+1}}\text{Log} (\Gamma(z))
%=(-1)^{m+1}m!\sum_{n=0}\frac{1}{(z+n)^{m+1}}
.$$
Therefore, \eqref{hurweq} can be rewritten as
$$L_f(\varphi_s^0)=
\frac{i^{2-s}}{(-s)!}\int_{i}^{i+1} f(z) \psi^{(-s)}(z).
dz$$
For $s=0$ (in which case, the proof of the corollary does not apply), this coincides with the  identity of Prop. \ref{Zag}.
\end{remark}

\begin{proof}[Proof of Theorem~\ref{mainInt*}]
Theorem~\ref{mainInt*} now follows immediately upon combining Prop. \ref{Zag}, Cor. \ref{bernWHF}, Cor. \ref{polyl} and Cor. \ref{hurw}.
\end{proof}

\section{Compactly supported values of the $L$-series}\label{comp}
As an extension in a different direction, we compute the values of $L_f$ at real-analytic compactly-supported functions $\varphi$.
Let $f\in S_k^!(N)$ with Fourier expansion $$f(z)=\sum_{n \ge -n_0} a_f(n) e^{2 \pi in z}.$$ Also, let $\varphi$ be a real-analytic function supported in $[a,b]$ so that, if $y \in [a, b]$, $\varphi(y)=\Phi(iy)$ for a holomorphic $\Phi\colon \mathbb H \to \mathbb C$. Further assume that
\begin{equation}\label{decay}
|\Phi(z)|<|z|^{-1-\varepsilon} \qquad \text{for $\Im(z) \in [a, b].$}
%\int_a^b|\Phi(T+it)|dt=o(1),\quad\quad (T\rightarrow\infty).
\end{equation}
%For instance, $|\Phi(z)|<|z|^{-1-\varepsilon}$, for $\Im(z) \in [a, b]$ is sufficient.

Since $\varphi \in \mathcal F_f$, by Lemma~\ref{lem:LfLf'_int} we have
\[
L_f(\varphi)=\int_0^{\infty}f(iy)\varphi(y)dy=\int_a^bf(iy)\varphi(y)dy.
\]
Then we get,
\[
\begin{aligned}
\int_a^bf(iy)\Phi(iy)dy&=-i\int_{ia}^{ib}f(z)\Phi(z)dz
=-i\left(\int_{ia}^{ia+T}+\int_{ia+T}^{ib+T}+\int_{ib+T}^{ib}\right)f(z)\Phi(z)dz
\\
&=-i\left(\int_{ia}^{ia+T}-\int_{ib}^{ib+T}\right)f(z)\Phi(z)dz+I_T,
\end{aligned}
\]
where
\[
I_T:=-i \int_{ia+T}^{ib+T} f(z)\Phi(z)dz=\int_a^bf(T+it)\Phi(T+it)dt.
\]
By the periodicity of $f$ we have
\[
|I_T|=O \left (\int_a^b|\Phi(T+it)|dt \right )
\]
with the implied constant independent of $T.$ This, with \eqref{decay}, implies
\[L_f(\varphi)=
-i\left(\int_{ia}^{ia+\infty}-\int_{ib}^{ib+\infty}\right)f(z)\Phi(z)dz.
\]
The first integral is
\[
\int_{ia}^{ia+\infty}f(z)\Phi(z)dz
=
\sum_{n\geq0}\int_{ia}^{ia+1}f(z+n)\Phi(z+n)dz
=
\int_{ia}^{ia+1}f(z)\widetilde \Phi(z+n)dz,
\]
where
$$\widetilde \Phi(z):=\sum_{n\geq0}\Phi(z+n).$$
By \eqref{decay}, this converges for all $\Im(z) \in [a, b].$
Repeating the computation for the integral depending on $b$, we deduce
\begin{proposition} Let $f \in S_k^!(N).$ Let $\varphi$ be a real-analytic function supported in $[a, b]$ such that, for all $y \in [a, b],$ $\varphi(y)=\Phi(iy)$ for some holomorphic $\Phi\colon \mathbb H \to \mathbb C$ satisfying \eqref{decay}. Then
\[
L_{f}(\varphi)=-i\left ( \int_{ia}^{ia+1}f(z)\widetilde\Phi(z)-\int_{ib}^{ib+1}f(z)\widetilde\Phi(z)dz \right ).
\]
\end{proposition}

\iffalse
%: Reference
\thispagestyle{empty}
{\footnotesize
\nocite{*}
\bibliographystyle{amsalpha}
\bibliography{reference}
}
\fi

\end{document}